\documentclass[12pt,twoside]{article}
\usepackage[latin1]{inputenc}%les accents
\usepackage{amsfonts,amsmath}% les maths
\usepackage[top=25mm, bottom=25mm, left=25mm, right=25mm]{geometry}
\newtheorem{thm}{Theorem}[section]

\newtheorem{lem}{Lemma}[section]
\newtheorem{prop}{Proposition}[section]
\newtheorem{exemp}{Example}[section]
\newtheorem{rem}{Remark}[section]
\numberwithin  {equation}{section}

\newcommand{\fact}[1]{#1\mathpunct{}!}
\def\R{\mathbb{R}}

\def\N{\mathbb{N}}

{\markboth{\sl  C. Abdelkefi, S. Chabchoub and
F. Rached}{\sl Integral remainder of order $k$ for Besov-Dunkl
spaces}

\begin{document}
\title{\bf Generalized Taylor formula with integral remainder for Besov-Dunkl spaces}
\author{Chokri Abdelkefi, Safa Chabchoub and Faten Rached
 \footnote{\small This work was
completed with the support of the DGRST research project LR11ES11,
University of Tunis El Manar.}\\ \small  Department of Mathematics,
Preparatory
Institute of Engineer Studies of Tunis  \\ \small 1089 Monfleury Tunis, University of Tunis, Tunisia\\
  \small E-mail : chokri.abdelkefi@yahoo.fr \\  \small E-mail : safachabchoub@yahoo.com
  \\  \small E-mail : rached@math.jussieu.fr}%
\date{}
\maketitle
\begin{abstract}
In the present paper, we propose to prove some properties and
estimates of the integral remainder in the generalized Taylor
formula associated to the Dunkl operator on the real line and to
describe the Besov-Dunkl spaces  for which the remainder has a given
order.
\end{abstract}
{\bf \small Key-words} : {\small Dunkl operator, Dunkl transform,
Dunkl translation operators, Dunkl convolution, Besov-Dunkl spaces.}\\
{\bf \small  MSC (2010)} : {\small 44A15, 44A35, 46E30}.
\section{Introduction}
$ $

On the real line, the Dunkl operator is a differential-difference
operator introduced in 1989, by C. Dunkl in \cite{dun} and is
denoted by $\Lambda_{\alpha}$ where $\alpha$ is a real parameter
$>-\frac{1}{2}$. The operator $\Lambda_{\alpha}$ plays a major role
in the study of quantum harmonic oscillators governed by Wigner's
commutation rules (see \cite{rose}). This operator is associated
with the reflexion group $ \mathbb{Z}_{2}$ on $\mathbb{R}$ and is
given by
$$ \Lambda_\alpha (f)(x) = \frac{df}{dx} (x) + \frac{2\alpha+1}{x}
 \Big[\frac{f(x)-f(-x)}{2}\Big],\; f \in \mathcal{C}^{1}(\R).$$
The Dunkl kernel $E_{\alpha}$ related to $\Lambda_{\alpha}$ is used
to define the Dunkl transform $\mathcal{F}_{\alpha}$ which enjoys
properties similar to those of the classical Fourier transform. The
Dunkl kernel $E_{\alpha}$
 satisfies a product formula (see \cite{ro1}). This allows us to define
the Dunkl translation $\tau_{x}$, $x\in\mathbb{R}$. As a result, we
have the Dunkl convolution $\ast_\alpha$ (see next section).

The classical Taylor formula with integral remainder was extended to
the one dimensional Dunkl operator $\Lambda_{\alpha}$ in \cite{mou}.
For $k=1,2,...,$ $f \in \mathcal{E}(\mathbb{R})$ and $ a \in \R$, we
have
\begin{eqnarray*} \tau _x(f)(a) = \sum_{p=0}^{k-1} b_p(x) \Lambda_\alpha^p f(a) +  R_k(x,f)(a),\quad x \in \mathbb{\mathbb{R}}\backslash\{0\} ,\end{eqnarray*}
with $R_k(x,f)(a)$ is the integral remainder of order $k$ given by
 \begin{eqnarray*} \displaystyle R_k(x,f)(a)= \int_{-|x|}^{|x|} \Theta_{k-1} (x,y) \tau_y (\Lambda_\alpha^{k} f)(a) A_\alpha(y) dy,\end{eqnarray*}
 where $\mathcal{E}(\mathbb{R})$ is the space of infinitely
differentiable functions on $\mathbb{R}$ and $ (\Theta_{p})_{p\in
\mathbb{N}}$, $(b_p)_{p\in \mathbb{N}}$ are two sequences of
 functions constructed inductively from the function $A_\alpha$ defined on $\mathbb{R}$ by $A_\alpha(x)=
 |x|^{2\alpha+1}$ (see next section).

Our aim in this paper is to describe the Besov-Dunkl spaces for
which the integral remainder in the generalized Taylor formula has a
given order.

Let $0<\beta <1$, $1 \leq p < +\infty $, $ 1 \leq q \leq +\infty$
and $k$ a positive integer $(k=1, 2,...)$. We denote by
$L^p(\mu_\alpha)$ the space of complex-valued functions $f$,
measurable on $\mathbb{R}$ such that
$$\|f\|_{p,\alpha} = \left(\int_{ \mathbb{R}}|f(x)|^p
d\mu_\alpha(x) \right)^{1/p} < + \infty,$$ where $\mu_\alpha$ is a
weighted Lebesgue measure associated with the Dunkl operator (see
next section). There are many ways to define the Besov spaces (see
\cite{Bes,Pe}) and the Besov-Dunkl spaces (see
\cite{ab1,ab2,ab3}):\\$\bullet$ The Besov-Dunkl space of order $k$
denoted
 by $\mathcal{B}^k\mathcal{D}_{p,q}^{\beta,\alpha}$
is the subspace of functions $f$ in $\mathcal{E}(\mathbb{R})\cap
L^p(\mu_\alpha)$ such that $\Lambda_\alpha^{k-1}(f) \in
L^p(\mu_\alpha)$ and satisfying
\begin{eqnarray*} \int_0^{+\infty} \Big(\frac{\omega_{p,\alpha}^k(x,f)}{x^{\beta+k-1}}\Big)^q
 \frac{dx}{x} < +\infty  &if& \ \  q < +\infty \\
\mbox{and} \qquad
    \sup_{x > 0} \frac{\omega_{p,\alpha}^k(x,f)}{x^{\beta+k-1}} < +\infty   &if& \ \  q =
  +\infty,\end{eqnarray*}
with $\omega_{p,\alpha}^k(x,f) =\displaystyle \sup_{|y| \leq x} \|
R_{k-1}(y, f)-b_{k-1}(y)\Lambda_\alpha^{k-1}f\|_{p,\alpha},$ where
we put for $k=1$, $\Lambda_\alpha^{0}f=f$ and $R_{0}(x, f)=\tau
_x(f).$
 \\$\bullet$ Put $\mathcal{D}_{p,\alpha}^k$ the subspace of functions $ f $ in
$\mathcal{E}(\mathbb{R})\cap L^p(\mu_\alpha)$ such that
  $ \Lambda_\alpha^k (f) \in L^p(\mu_\alpha)$.
We consider the subspace $\mathcal{K}^k\mathcal{D}_{p,q}^{\beta,\alpha}$ of functions
 $f \in \mathcal{D}_{p,\alpha}^{k-1}+\mathcal{D}_{p,\alpha}^{k} $ satisfying
 \begin{eqnarray*}
 \int_0^{+\infty} \Big(\frac{K_{p,\alpha}^k(x,f)}{x^{\beta}}\Big)^q \frac{dx}{x} < +\infty & if &  q <
 +\infty\\
\mbox{and}\qquad
   \sup_{x > 0} \frac{K_{p,\alpha}^k(x,f)}{x^{\beta}} < +\infty & if &  q = +\infty,
\end{eqnarray*} where $K_{p,\alpha}^k $  is the Peetre K-functional given by
$$K_{p,\alpha}^k(x,f)= \inf_{f=f_0+f_1}\Big\{\| \Lambda_\alpha^{k-1} (f_0) \|_{p,\alpha}+ x \| \Lambda_\alpha^k (f_1)
\|_{p,\alpha},\,
 f_0 \in \mathcal{D}_{p,\alpha}^{k-1},\, f_1 \in \mathcal{D}_{p,\alpha}^k\Big\}.$$
$\bullet$
$\widetilde{\mathcal{B}}^k\mathcal{D}_{p,q}^{\beta,\alpha}$ denote
the subspace of functions $f$ in $\mathcal{E}(\mathbb{R})\cap
L^p(\mu_\alpha)$ such that $\Lambda_\alpha^{k-1}(f)\in
L^p(\mu_\alpha)$ and satisfying
\begin{eqnarray*} \int_0^{+\infty} \Big(\frac{\widetilde{\omega}_{p,\alpha}^k(x,f)}{x^{\beta+k-1}}\Big)^q
 \frac{dx}{x} < +\infty  &if& \ \  q < +\infty \\
\mbox{and} \qquad
    \sup_{x > 0} \frac{\widetilde{\omega}_{p,\alpha}^k(x,f)}{x^{\beta+k-1}} < +\infty   &if& \ \  q =
  +\infty,\end{eqnarray*}
with $\widetilde{\omega}_{p,\alpha}^k(x,f) =\displaystyle \big\|
R_{k-1}(x, f)+R_{k-1}(-x,
f)-\big(b_{k-1}(x)+b_{k-1}(-x)\big)\Lambda_\alpha^{k-1}
f\big\|_{p,\alpha},$ where we put for $k=1$,
$\Lambda_\alpha^{0}f=f$, $R_{0}(x, f)=\tau _x(f)$ and $R_{0}(-x,
f)=\tau
_{-x}(f).$\\
$\bullet$ Let $ \phi \in \mathcal{S}_\ast(\mathbb{R})$ such that
$\displaystyle\int_0^{+\infty}x^{2i}\phi(x)d\mu_\alpha(x)=0 $, for
all $i\in \{0,1,...,[\frac{k-1}{2}]\}$ where $ \mathcal{S}_\ast(
\mathbb{R})$ is the space of even Schwartz functions
 on $\mathbb{R}$ (see Example 4.2, section 4). We shall
 denote by $ \mathcal{C}_{p,q}^{k,\beta,\alpha}$ the subspace of functions $f$ in $\mathcal{E}(\mathbb{R})$
such that $\Lambda_\alpha^{2i}(f)\in L^p(\mu_\alpha),$ $0\leq i\leq
[\frac{k-1}{2}]$ and satisfying
\begin{eqnarray*} \int^{+ \infty}_0 \left(\frac{\|f \ast_\alpha
\phi_t \|_{p,\alpha}}{t^{\beta+k-1}}\right)^q \, \frac{dt}{t} < +
\infty  &if& q < +\infty
\\
\mbox{and}\qquad \sup_{t>0}\frac{\|f \ast_\alpha
\phi_t\|_{p,\alpha}}{t^{\beta+k-1}} <+\infty  & if & q =
+\infty,\end{eqnarray*} where $\phi_t$ is the dilation of $\phi$
given by $\phi_t(x)=\frac{1}{t^{2(\alpha+1)}}\phi(\frac{x}{t})$, for
all $t\in (0,+\infty)$ and $x\in\mathbb{R}$.\\

In this paper, we give some properties and estimates of the integral
remainder of order $k$ and we establish that
$$\mathcal{B}^k\mathcal{D}_{p,q}^{\beta,\alpha}=\mathcal{K}^k\mathcal{D}_{p,q}^{\beta,\alpha}
\quad\mbox{;}\quad
\widetilde{\mathcal{B}}^k\mathcal{D}_{p,q}^{\beta,\alpha}=
\mathcal{C}_{p,q}^{k,\beta,\alpha}.$$ Note that we have
$\mathcal{B}^k\mathcal{D}_{p,q}^{\beta,\alpha}\subset
\widetilde{\mathcal{B}}^k\mathcal{D}_{p,q}^{\beta,\alpha}$ (see
section 4).

The results obtained in this paper are an extension to the Dunkl
theory on the real line of those obtained in \cite{An,L.P}. More
precisely, in \cite{An}, the authors showed in the classical case
and for $k=1$ that
$\widetilde{\mathcal{B}}^k\mathcal{D}_{p,q}^{\beta,\alpha}=
\mathcal{C}_{p,q}^{k,\beta,\alpha}.$\\

The contents of this paper are as follows. \\In section 2, we
collect some basic definitions and results about harmonic analysis
associated with the Dunkl operator $\Lambda_\alpha$. \\
In section 3, we prove some properties and estimates of the integral
remainder of order $k$.\\Finally, we establish in the section 4, the
coincidence between the different characterizations of the
Besov-Dunkl spaces.

Along this paper, we use $c$ to represent a suitable positive
constant which is not necessarily the same in each occurrence.

\section{Preliminaries}

\label{sec:1} In this section, we recall some notations and results
in Dunkl theory on $\mathbb{R}$ and we refer for more details to \cite{am,dun,ro1}.\\

 For $\lambda \in \mathbb{C}$, the initial problem
$$\Lambda_\alpha(f)(x) = \lambda f(x),\quad f(0) = 1,\quad x \in \mathbb{R},$$
has a unique solution $E_\alpha(\lambda .)$ called Dunkl kernel
given by
$$E_\alpha(\lambda x) = j_\alpha(i\lambda x) + \frac{\lambda x}
{2(\alpha+1)} j_{\alpha+1} (i\lambda x),\quad x \in \mathbb{R},$$
where $j_\alpha$ is the normalized Bessel function of the first kind
and order $\alpha$, defined by
$$j_\alpha(\lambda x) = \left\{ \begin{array}{ll}
2^\alpha \Gamma (\alpha +1) \,\frac{J_\alpha(\lambda x)}{(\lambda
x)^\alpha} &\mbox{ if } \lambda x \neq 0\\
1 &\mbox{ if } \lambda x = 0 ,
\end{array}\right.$$
here $J_\alpha$ is the Bessel function of first kind and order
$\alpha$.\\ We have for all $x \in \mathbb{R}$, the function
$\lambda \rightarrow j_\alpha (\lambda x)$ is even on $\mathbb{R}$
\\and $$|E_\alpha(-i\lambda x)| \leq 1.$$ Let $A_\alpha$ the
function defined on $\mathbb{R}$ by
$$A_\alpha(x) = |x|^{2\alpha+1},\quad x \in \mathbb{R},$$ and $\mu_\alpha$ the weighted
Lebesgue measure on $\mathbb{R}$ given by
\begin{eqnarray}d\mu_\alpha(x) = \frac{A_\alpha(x)}{2^{\alpha +1}\Gamma(\alpha
+1)}dx.\end{eqnarray} For every $1 \leq p \leq + \infty$, we denote
by $L^p(\mu_\alpha)$ the space of complex-valued functions $f$,
measurable on $\mathbb{R}$ such that
$$\|f\|_{p,\alpha} = \left(\int_{ \mathbb{R}}|f(x)|^p
d\mu_\alpha(x) \right)^{1/p} < + \infty,\quad \mbox{ if } p < +
\infty,$$ and $$\|f\|_{\infty} = ess\sup_{x \in \mathbb{R}}|f(x)| <
+ \infty.$$ There exists an analogue of the classical Fourier
transform with respect to the Dunkl kernel called the Dunkl
transform and denoted by $\mathcal{F}_\alpha$. The Dunkl transform
enjoys properties similar to those of the classical Fourier
transform and is defined for $f \in L^1(\mu_\alpha)$ by
\begin{eqnarray*}
\mathcal{F}_\alpha(f)(x) = \int_{\mathbb{R}}f(y)\,E_\alpha(-ixy)\,
d\mu_\alpha(y), \quad x \in \mathbb{R}.
\end{eqnarray*}
For all $x, y, z \in \mathbb{R}$, we consider
 $$W_\alpha(x,y,z) = \frac{(\Gamma(\alpha+1)^2)}{2^{\alpha-1}
 \sqrt{\pi}\Gamma(\alpha + \frac{1}{2})}
  (1 - b_{x,y,z} + b_{z,x,y} +
 b_{z,y,x}) \Delta_\alpha(x,y,z)$$
where
$$b_{x,y,z} = \left\{ \begin{array}{ll}
\frac{x^2 + y^2 -z^2}{2xy} &\mbox{ if } x, y \in \mathbb{R}
\backslash \{0\},\; z \in \mathbb{R}\\
0 &\mbox{ otherwise }
\end{array}\right.$$
and
$$\Delta_\alpha(x,y,z) = \left\{ \begin{array}{ll}
\frac{([(|x| + |y|)^2 - z^2][z^2-(|x| - |y|)^2])^{\alpha -
\frac{1}{2}}}{|xyz|^{2\alpha}} &\mbox{ if } |z|\in S_{x,y}\\
0 &\mbox{ otherwise }
\end{array}\right.$$
where $$S_{x,y} = \Big[||x| - |y||\;,\; |x| + |y|\Big].$$ The kernel
$W_\alpha $, is even and we have
$$W_\alpha(x,,y,z) = W_\alpha(y,x,z) = W_\alpha(-x,z,y) =
W_\alpha(-z, y, -x)$$ and
$$\int_{\mathbb{R}}|W_\alpha(x,y,z)|d\mu_\alpha(z) \leq
\sqrt{2}.$$
 The Dunkl kernel $E_\alpha$ satisfies the following
product formula
$$E_\alpha(ixt) E_\alpha(iyt) = \int_{\mathbb{R}} E_\alpha(itz)
d\gamma_{x,y}(z),\quad x, y, t \in \mathbb{R}, $$ where
$\gamma_{x,y}$ is a signed measure on $\mathbb{R}$ given by
\begin{eqnarray}d\gamma_{x,y}(z) = \left\{
\begin{array}{ll} W_\alpha(x,y,z)d\mu_\alpha(z) &\mbox{ if } x, y
\in \mathbb{R} \backslash \{0\}\\ d\delta_x(z) &\mbox{ if } y = 0\\
d\delta_y(z) &\mbox{ if } x = 0.
\end{array}\right.\end{eqnarray} with
$\mbox{supp}\gamma_{x,y} = S_{x,y} \cup (-S_{x,y}).$\\ For $x, y \in
\mathbb{R}$ and $f$ a continuous function on $\mathbb{R}$, the Dunkl
translation operator $\tau_x$ given by
 $$\tau_x(f)(y) =\int_{\mathbb{R}} f(z) d\gamma_{x,y}(z)$$ satisfies the
following properties :
\begin{itemize}
\item $\tau_x$ is a continuous linear operator from
$\mathcal{E}( \mathbb{R})$ into itself.
\item For all $f \in
\mathcal{E}(\mathbb{R})$,  we have \begin{eqnarray}\tau_x(f)(y) =
\tau_y(f)(x)\quad\mbox{and}\quad \tau_0(f)(x) = f(x)\end{eqnarray}
\begin{eqnarray}\tau_x \,o\, \tau_y = \tau_y\,o\,\tau_x\quad\mbox{and}\quad
\Lambda_\alpha \,o\,\tau_x = \tau_x \,o\,\Lambda_\alpha
.\end{eqnarray}
\item For all $x \in \mathbb{R}$, the
operator $\tau_x$ extends to $L^p(\mu_\alpha),\; p \geq 1$ and we
have for $f \in L^p(\mu_\alpha)$
\begin{eqnarray}
\|\tau_x(f)\|_{p,\alpha} \leq \sqrt{2} \|f\|_{p,\alpha}.
\end{eqnarray}
\end{itemize}
The Dunkl convolution  $f\, \ast_\alpha g$ of two continuous
functions $f$ and $g$ on $\mathbb{R}$ with compact support, is
defined by
$$(f\,\ast_\alpha\, g)(x) = \int_{\mathbb{R}} \tau_x(f)(-y) g(y)
d\mu_\alpha(y),\quad x \in \mathbb{R}. $$ The convolution
$\ast_\alpha$ is associative and commutative and satisfies the
following property:
\begin{itemize}
\item Assume that $p,q, r \in [1, + \infty[$ satisfying
$\frac{1}{p} + \frac{1}{q} = 1 + \frac{1}{r}$ (the Young condition).
Then the map $(f, g) \rightarrow f\,\ast_\alpha \,g$ defined on
$C_c(\mathbb{R}) \times C_c( \mathbb{R})$, extends to a continuous
map from $L^p(\mu_\alpha) \times L^q(\mu_\alpha)$ to
$L^r(\mu_\alpha)$ and we have \begin{eqnarray}\|f\,\ast_\alpha
\,g\|_{r,\alpha} \leq \sqrt{2}\|f\|_{p,\alpha}
\|g\|_{q,\alpha}.\end{eqnarray}
\item For all $f
\in L^1(\mu_\alpha)$, $g \in L^2(\mu_\alpha)$ and $h \in
L^p(\mu_\alpha)$, $1 \leq p < +\infty,$ we have
\begin{eqnarray}  \mathcal{F}_\alpha(f\,\ast_\alpha g) =
\mathcal{F}_\alpha(f) \mathcal{F}_\alpha(g)\quad\mbox{and}\quad
 \tau_t (f\,\ast_\alpha\,h) = \tau_t(f)\,\ast_\alpha h =
f\,\ast_\alpha \tau_t(h),\; t \in \mathbb{R}.\end{eqnarray}
\end{itemize}
It has been shown in \cite{mou}, the following generalized Taylor
formula with integral remainder:\begin{prop} For $k=1,2,...,$ $f \in
\mathcal{E}(\mathbb{R})$ and $ a \in \R$, we have
\begin{eqnarray} \tau _x f(a) = \sum_{p=0}^{k-1} b_p(x) \Lambda_\alpha^p f(a) +  R_k(x,f)(a),\quad x \in \R\backslash\{0\} ,\end{eqnarray}
with $ R_k(x,f)(a)$ is the integral remainder of order $k$ given by
 \begin{eqnarray} \displaystyle R_k(x,f)(a)= \int_{-|x|}^{|x|} \Theta_{k-1} (x,y) \tau_y (\Lambda_\alpha^{k} f)(a) A_\alpha(y) dy,\end{eqnarray}
where\begin{itemize}
\item[i)]  $\displaystyle b_{2m}(x)= \frac{1}{(\alpha+1)_m \fact{m}}  \Big(\frac{x}{2} \Big)^{2m}\;$, $\;\displaystyle
  b_{2m+1}(x)= \frac{1}{(\alpha+1)_{m+1} \fact{m}}  \Big(\frac{x}{2}
  \Big)^{2m+1}$, for all $\;m\in \N.$
 \item[ii)] $ \Theta_{k-1}(x,y) = u_{k-1}(x,y) + v_{k-1}(x,y)\;$ with
   $\;\displaystyle u_0(x,y)= \frac{sgn(x)}{2 A_\alpha(x)}\,$ , $\,\displaystyle v_0(x,y)= \frac{sgn(y)}{2
 A_\alpha(y)},$\\ and
 $\quad\displaystyle u_k(x,y)= \int_{|y|}^{|x|} v_{k-1}(x,z) dz\;$ , $\;\displaystyle
 v_k(x,y)= \frac{sgn(y)}{ A_\alpha(y)}\int_{|y|}^{|x|} u_{k-1}(x,z)A_\alpha(z)
   dz.$
   \end{itemize}
   \end{prop}
   According to (\cite{ro3}, Lemma 2.2), the Dunkl operator
   $\Lambda_\alpha$ have the following regularity properties:
\begin{eqnarray} \Lambda_\alpha \;\mbox{leaves}\;\, \mathcal{C}_c^\infty(\mathbb{R})\;
 \mbox{and} \;\mbox {the\, Schwartz\, space}\; \mathcal{S}(\mathbb{R})\; \mbox{invariant}. \end{eqnarray}
\section{Some properties of the integral remainder of order $k$} In this section,
we prove some properties and estimates of the integral remainder in
the generalized Taylor formula.
\begin{rem}
 Let $k=1,2,...,$ $f \in \mathcal{E}(\mathbb{R})$ and $ x \in
 \R\backslash\{0\}$.
 \begin{enumerate}
\item From Proposition 2.1, we have
  \begin{eqnarray}
R_k(x, f)& =& \tau_x(f)- f-b_1(x)\Lambda_\alpha f...- b_{k-1}(x)\Lambda_\alpha^{k-1}f \nonumber  \\
  &=&R_{k-1}(x, f)-b_{k-1}(x)\Lambda_\alpha^{k-1} f,
  \end{eqnarray} where we put for $k=1$, $R_{0}(x, f) = \tau_x(f).$
  Observe that $$R_1(x, f)=R_{0}(x, f)-b_{0}(x)\Lambda_\alpha^{0} f=\tau_x(f)-
  f.$$
  \item According to (\cite{mou}, p.352) and Proposition 2.1, i), we have
 \begin{eqnarray}
 \displaystyle \int_{-|x|}^{|x|} |\Theta_{k-1} (x,y)|  A_\alpha(y) dy &\leq& b_k(|x|)+|x| b_{k-1}(|x|)\nonumber \\
  &\leq& c\, |x|^{k}.
\end{eqnarray}
\item Note that the function $y\longmapsto
\tau_y(f)-f$ is continuous on $\mathbb{R}$ (see \cite{mou.T}, Lemma
1, (ii)), which implies that the same is true for the function
$y\longmapsto R_k(y, f).$
 \end{enumerate}
\end{rem}
\begin{lem}
 Let $k=1,2,...,$ and $f \in \mathcal{E}(\mathbb{R})$ such that $\Lambda_\alpha^{k-1}f \in L^p(\mu_\alpha) $. Then we have
 \begin{equation}
  \| R_{k-1}(x,f)\rVert_{p,\alpha} \leq c \,|x|^{k-1} \| \Lambda_\alpha^{k-1} f \rVert_{p,\alpha},\quad x \in \R\backslash\{0\}.
 \end{equation}
 \end{lem}
\begin{proof}
 Let $k=1,2,...,$ $f \in \mathcal{E}(\mathbb{R})$ such that $\Lambda_\alpha^{k-1}f \in L^p(\mu_\alpha) $ and $x \in
 \R\backslash\{0\}$. For $k=1$, by (2.5), it's clear that $\|R_{0}(x, f) = \tau_x(f)\|_{p,\alpha}\leq c\,\|f\|_{p,\alpha}.$
Using the Minkowski's inequality for integrals, (2.5) and (2.9), we
have for $k\geq2$
 \begin{eqnarray*}
  \| R_{k-1}(x,f)\|_{p,\alpha}  &\leq& \int_{-|x|}^{|x|}| \Theta_{k-2} (x,y)| \
   \| \tau_y ( \Lambda_\alpha^{k-1} f)\|_{p,\alpha} A_\alpha(y) dy\\
   &\leq& c \;  \|   \Lambda_\alpha^{k-1} f \|_{p,\alpha} \int_{-|x|}^{|x|} |\Theta_{k-2} (x,y)| A_\alpha(y) dy.
 \end{eqnarray*}
Using (3.2),  we deduce our result.
\end{proof}
\begin{rem} Let $k=1,2,...,$ $f \in \mathcal{E}(\mathbb{R})$ such that $\Lambda_\alpha^{k-1}f \in L^p(\mu_\alpha) $ and
$x \in \mathbb{R}\backslash\{0\}.$ Then we have by (3.1), (3.3) and
Proposition 2.1, i),
 \begin{eqnarray}
 \|R_k(x,f)\|_{p,\alpha}&=&\|R_{k-1}(x,f)+b_{k-1}(x) \Lambda_\alpha^{k-1} f\|_{p,\alpha}\nonumber\\&\leq&
\|R_{k-1}(x,f)\|_{p,\alpha}+
    \|b_{k-1}(x) \Lambda_\alpha^{k-1} f\|_{p,\alpha} \nonumber
     \\&\leq& c \,|x|^{k-1} \|\Lambda_\alpha^{k-1} f \|_{p,\alpha}.
 \end{eqnarray}
\end{rem}
\begin{lem}
For $x \in \R\backslash\{0\}$ and $p\in \N$, we have
\begin{equation}
 \int_{-|x|}^{|x|} \Theta_0 (x,y)  b_p(y)  A_\alpha(y) dy = b_{p+1}(y).
\end{equation}
\end{lem}
\begin{proof}
 Let $x\in \R\backslash\{0\} $. Using Proposition 2.1, we have
 \begin{itemize}
\item  If $p=2m,$ $m\in
 \N,$
  \begin{align*}
   \int_{-|x|}^{|x|} \Theta_0 (x,y)  b_{2m}(y)  A_\alpha(y) dy
   & = \int_{-|x|}^{|x|} u_0(x,y) b_{2m}(y) A_\alpha(y) dy + \int_{-|x|}^{|x|} v_0(x,y) b_{2m}(y) A_\alpha(y) dy \\
   & = \int_{-|x|}^{|x|} \frac{sgn(x)|y|^{2\alpha+1}}{2|x|^{2\alpha+1}} b_{2m}(y) dy + \int_{-|x|}^{|x|} \frac{sgn(y)}{2} b_{2m}(y)  dy \\
   &=  \frac{x}{2^{2m}|x|^{2\alpha+2}(\alpha+1)_m \fact{m}}\int_{0}^{|x|} y^{2\alpha+2m+1} dy\\
   &= \frac{x}{2^{2m}(\alpha+1)_m \fact{m}} \frac{|x|^{2m}}{2 (\alpha +m+1)} \\
   &= b_{2m+1}(x).
  \end{align*}
\item If $p =2m+1,$ $m\in \N,$ we get
  \begin{align*}
   \int_{-|x|}^{|x|}  \Theta_0 (x,y)  b_{2m+1}(y)  A_\alpha(y) dy
   & =\int_{-|x|}^{|x|} u_0(x,y) b_{2m+1}(y) A_\alpha(y) dy + \int_{-|x|}^{|x|} v_0(x,y) b_{2m+1}(y) A_\alpha(y) dy \\\\
   & = \int_{-|x|}^{|x|} \frac{sgn(x)|y|^{2\alpha+1}}{2|x|^{2\alpha+1}} b_{2m+1}(y) dy + \int_{-|x|}^{|x|} \frac{sgn(y)}{2} b_{2m+1}(y)  dy \\
   &=  \frac{1}{2^{2m+1}(\alpha+1)_{m+1} \fact{m}}\int_{0}^{|x|} y^{2m+1} dy\\
   &= \frac{1}{2^{2m+1}(\alpha+1)_{m+1} \fact{m}} \frac{|x|^{2m+2}}{2 (m+1)} \\
   &= b_{2m+2}(x).
  \end{align*}
  \end{itemize}
Our Lemma is proved.
\end{proof}
\begin{lem}
  Let $k=1,2,...,$ $f \in \mathcal{E}(\mathbb{R})$, $x \in \R\backslash\{0\}$ and $a \in
  \R$. Then we have,
 \begin{eqnarray}
R_k(x, f)(a) =\int_{-|x|}^{|x|} \Theta_0(x,y)
R_{k-1}(y,\Lambda_\alpha f)(a) A_\alpha(y) dy.
\end{eqnarray}
\end{lem}
\begin{proof}
Let $k=1,2,...,$ $f \in \mathcal{E}(\mathbb{R})$, $x \in
\R\backslash\{0\}$ and $a \in \R$.
\begin{itemize}
 \item We have from (2.8), (2.9) and the fact that $R_0(y,\Lambda_\alpha f)(a)= \tau_y(\Lambda_\alpha  f)$
$$  R_1(x,f)(a) = (\tau_x(f)-f)(a)=\int_{-|x|}^{|x|} \Theta_0(x,y) \tau_y(\Lambda_\alpha f)(a) A_\alpha(y)
dy,
$$ hence the property (3.6) is true for $k=1.$
  \item Suppose that
  $$ R_k(x,f)(a) =\int_{-|x|}^{|x|} \Theta_0(x,y) R_{k-1}(y,\Lambda_\alpha f)(a) A_\alpha(y) dy, $$
   then by (3.1) and (3.5) again, we get
 \begin{align*}
  \int_{-|x|}^{|x|} \Theta_0(x,y) R_k(y,\Lambda_\alpha f)(a) A_\alpha(y) dy &=
  \int_{-|x|}^{|x|} \Theta_0(x,y) [ R_{k-1}(y,\Lambda_\alpha f)-b_{k-1}(y)\Lambda_\alpha^{k} f ](a)A_\alpha(y) dy \\
  &= \int_{-|x|}^{|x|} \Theta_0(x,y)R_{k-1}(y,\Lambda_\alpha f)(a)A_\alpha(y) dy \\
  &- \int_{-|x|}^{|x|} \Theta_0(x,y) b_{k-1}(y)\Lambda_\alpha^{k} f(a)A_\alpha(y) dy\\
  &= R_k(x,f)(a) - \Lambda_\alpha^k f(a)\int_{-|x|}^{|x|} \Theta_0(x,y) b_{k-1}(y)A_\alpha(y) dy \\
  &= R_k(x,f)(a) - b_k(x)\Lambda_\alpha^k f(a)\\
  &= R_{k+1}(x, f)(a).
  \end{align*}
\end{itemize}Hence by induction, we deduce our result.
\end{proof}
\begin{lem}
Let $k=1,2,...,$ $f \in \mathcal{E}(\mathbb{R})$, $x \in
\R\backslash\{0\}$ and $a \in
  \R$. We denote by
 $$ \displaystyle I_1(x,f)(a) = \int_{-|x|}^{|x|} \Theta_0 (x,y)
  \tau_y (f)(a) A_\alpha(y) dy,$$
  and for  $k \geq 2$
  $$\displaystyle I_k(x,f)(a) = \int_{-|x|}^{|x|} \Theta_0 (x,y)
  I_{k-1}(y,f) (a) A_\alpha(y) dy.$$
Then, we have
\begin{eqnarray}
\Lambda_\alpha^{k+1} \big(I_k(x,f)\big)(a)&=&\Lambda_\alpha^{k}
\big(I_k(x,\Lambda_\alpha f)\big)(a),\\ \mbox{and} \qquad
\Lambda_\alpha^k I_k(x,f)(a) &=& R_k(x,f)(a).
\end{eqnarray}
   \end{lem}
\begin{proof} Let $k=1,2,...,$ $f \in \mathcal{E}(\mathbb{R})$, $x \in
\mathbb{R}\backslash\{0\}$ and $a \in
  \mathbb{R}$.
  \begin{itemize}
  \item Using (2.4), we have \begin{eqnarray*}
\Lambda_\alpha^{2} \big(I_1(x,f)\big)(a)&= &  \int_{-|x|}^{|x|}
\Theta_0 (x,y)
  \Lambda_\alpha\tau_y (\Lambda_\alpha f)(a) A_\alpha(y) dy\\
  &=&\Lambda_\alpha\big( I_1(x,\Lambda_\alpha f)\big)(a).\end{eqnarray*}
    Suppose that $$\Lambda_\alpha^{k+1}
    \big(I_k(x,f)\big)(a)=\Lambda_\alpha^{k}\big(I_k(x,\Lambda_\alpha f)\big)(a),$$
then we have \begin{eqnarray*}\Lambda_\alpha^{k+2}
\big(I_{k+1}(x,f)\big)(a)&=& \int_{-|x|}^{|x|} \Theta_0 (x,y)
 \Lambda_\alpha \big(\Lambda_\alpha^{k+1}
    I_k(y,f)\big)(a) A_\alpha(y) dy\\&=& \int_{-|x|}^{|x|} \Theta_0 (x,y)
 \Lambda_\alpha \big(\Lambda_\alpha^{k}
    I_k(y,\Lambda_\alpha f)\big)(a) A_\alpha(y) dy\\&=& \Lambda_\alpha^{k+1}
\big(I_{k+1}(x,\Lambda_\alpha f)\big)(a),\end{eqnarray*}
 hence by induction, we obtain our result.
\item From (2.4), (2.9) and (3.6), we can write
 \begin{align*}
   \Lambda_\alpha \big(I_1(x,f)\big)(a) &=    \int_{-|x|}^{|x|} \Theta_0 (x,y) \Lambda_\alpha(\tau_y f)(a) A_\alpha(y) dy\\
&=  \int_{-|x|}^{|x|} \Theta_0 (x,y)\tau_y (\Lambda_\alpha f)(a)A_\alpha(y) dy\\
&= R_1(x,f)(a).
   \end{align*}
    Suppose that
   $$\Lambda_\alpha^k\big( I_k(x,f)\big)(a) = R_k(x,f)(a),$$
    then by (3.6) and (3.7), we have
    \begin{align*}
     \Lambda_\alpha^{k+1} \big(I_{k+1}(x,f)\big)(a)&= \int_{-|x|}^{|x|} \Theta_0 (x,y) \Lambda_\alpha^{k+1}
      \big(I_k(y,f)\big)(a)A_\alpha(y) dy\\
     &=  \int_{-|x|}^{|x|} \Theta_0 (x,y) \Lambda_\alpha^k \big(I_k(y,\Lambda_\alpha f)\big)(a)A_\alpha(y) dy\\
     &=  \int_{-|x|}^{|x|} \Theta_0 (x,y) R_k(y,\Lambda_\alpha f)(a)A_\alpha(y) dy\\
     &= R_{k+1}(x,f)(a).
    \end{align*}
\end{itemize} By induction, we deduce our result.
 \end{proof}
 \begin{rem} For $k=1,2,...,$ $f \in \mathcal{E}(\mathbb{R})$ and $x \in
\R\backslash\{0\}$, we observe from Proposition 2.1 that
 \begin{eqnarray}R_{k}(x,f)+R_{k}(-x,f)&=&\tau_x(f)+\tau_{-x}(f)-\sum_{p=0}^{k-1}
  \big(b_p(x)+b_p(-x)\big)\Lambda_\alpha^p f\nonumber
  \\&=&\tau_x(f)+\tau_{-x}(f)-2
  \sum_{i=0}^{[\frac{k-1}{2}]} b_{2i}(x)\Lambda_\alpha^{2i} f.\end{eqnarray}
\end{rem}
\section{Characterizations of Besov-Dunkl spaces of order $k$}
$ $ In this section, we establish respectively that
$\mathcal{B}^k\mathcal{D}_{p,q}^{\beta,\alpha}=\mathcal{K}^k\mathcal{D}_{p,q}^{\beta,\alpha}$
and $\widetilde{\mathcal{B}}^k\mathcal{D}_{p,q}^{\beta,\alpha}=
\mathcal{C}_{p,q}^{k,\beta,\alpha}.$ \\Before proving these results,
we begin with a useful remarks, a proposition containing sufficient
conditions and an example.
\begin{rem} For $k=1, 2,...,$ $f\in \mathcal{E}(\mathbb{R})$ such that
$\Lambda_\alpha^{k-1}(f) \in L^p(\mu_\alpha)$ and $x\in(0,+\infty),$
we can assert from (3.1) and (3.4) that
\begin{itemize}
  \item[1/] $\omega_{p,\alpha}^k(x,f)=\displaystyle\sup_{|y| \leq
x} \| R_{k}(y,f)\|_{p,\alpha}.$
\item[2/] $\widetilde{\omega}_{p,\alpha}^k(x,f) =\| R_{k}(x, f)+R_{k}(-x,
f)\|_{p,\alpha}.$
\end{itemize}
\end{rem}
\begin{prop} Let $ 1 \leq p < +\infty $, $ 1 \leq q \leq +\infty$,
$0<\beta<1$, $f\in \mathcal{E}(\mathbb{R})$ and $k=1, 2,...$. If
$\Lambda_\alpha^{k-1}(f)$ and $\Lambda_\alpha^k (f)$ are in
$L^p(\mu_\alpha)$, then
$f\in\mathcal{B}^k\mathcal{D}^{p,q}_{\beta,\alpha}.$
\end{prop}
\begin{proof} Let $ 1 \leq p < +\infty $, $ 1 \leq q \leq +\infty$,
$0<\beta<1$ and $f\in \mathcal{E}(\mathbb{R})$ such that
$\Lambda_\alpha^{k-1}(f),\,\Lambda_\alpha^k (f)$ are in
$L^p(\mu_\alpha)$ for $k=1, 2,....$
 By (3.3) and (3.4), we obtain for $x \in (0,+\infty)$
\begin{eqnarray*}\omega_{p,\alpha}^k(x,f)\leq c \,x^k \lVert \Lambda_\alpha^k f
\rVert_{p,\alpha} \quad\mbox{and}\quad \omega_{p,\alpha}^k(x,f)\leq
c \,x^{k-1} \lVert \Lambda_\alpha^{k-1} f
\rVert_{p,\alpha}.\end{eqnarray*} Then we can write,
\begin{eqnarray*}
\int_0^{+\infty}
\Big(\frac{\omega_{p,\alpha}^k(x,f)}{x^{\beta+k-1}}\Big)^q
\frac{dx}{x}\leq c \int_0^1 \Big(\frac{\| \Lambda_\alpha^k f
\|_{p,\alpha}}{x^{\beta-1}}\Big)^q \frac{dx}{x}+c \int_1^{+\infty}
\Big(\frac{\| \Lambda_\alpha^{k-1} f
\|_{p,\alpha}}{x^{\beta}}\Big)^q \frac{dx}{x}<+\infty.
\end{eqnarray*} Here when $q=+\infty$, we make the usual
modification.
\end{proof}
\begin{exemp} From (2.10) and Proposition 4.1, we can assert that the
spaces $\mathcal{C}_c^\infty(\mathbb{R})$ and
$\mathcal{S}(\mathbb{R})$ are included in
$\mathcal{B}^k\mathcal{D}^{p,q}_{\beta,\alpha}.$
\end{exemp}
\begin{rem} By the fact that $\widetilde{\omega}_{p,\alpha}^k(x,f)\leq
2\,\omega_{p,\alpha}^k(x,f)$, we have clearly
$\mathcal{B}^k\mathcal{D}_{p,q}^{\beta,\alpha}\subset
\widetilde{\mathcal{B}}^k\mathcal{D}_{p,q}^{\beta,\alpha}.$ Observe
that for $k=1$, we have $$\omega_{p,\alpha}^k(x,f) =\displaystyle
\sup_{|y| \leq x} \| \tau_y(f)- f\|_{p,\alpha}\quad\mbox{and}\quad
\widetilde{\omega}_{p,\alpha}^k(x,f) = \| \tau_x(f)+\tau_{-x}(f)-
2f\|_{p,\alpha}.$$\end{rem}
 \begin{thm}
  Let $0 <\beta <1$, $ k=1,2,...,$ $ 1 \leq p < +\infty $ and $ 1 \leq q \leq +\infty$, then
  \begin{equation*}
   \mathcal{B}^k\mathcal{D}_{p,q}^{\beta,\alpha} =
   \mathcal{K}^k\mathcal{D}_{p,q}^{\beta,\alpha}.
  \end{equation*}
  \end{thm}
\begin{proof}
   We start with the proof of the inclusion
   $\mathcal{K}^k\mathcal{D}_{p,q}^{\beta,\alpha}\subset\mathcal{B}^k\mathcal{D}_{p,q}^{\beta,\alpha}$.
    For $ f \in \mathcal{K}^k\mathcal{D}_{p,q}^{\beta,\alpha}$, $f=f_0+f_1$, $f_0 \in \mathcal{D}_{p,\alpha}^{k-1}$
     and $f_1 \in \mathcal{D}_{p,\alpha}^k$,
    we have by (3.3)
 \begin{eqnarray}
   \omega_{p,\alpha}^k(x,f_1) &=&  \sup_{|y| \leq x} \| R_k(y,f_1)\|_{p,\alpha}\nonumber\\
   &\leq& c \  x^k \| \Lambda_\alpha^k f_1 \|_{p,\alpha},\quad  x
   \in (0,+\infty).
 \end{eqnarray}
Using (3.4), we obtain
  \begin{eqnarray}
   \omega_{p,\alpha}^k(x,f_0) &\leq& \sup_{|y| \leq x} \| R_{k-1}(y,f_0)\|_{p,\alpha}+
    \sup_{|y| \leq x} \| b_{k-1}(y) \Lambda_\alpha^{k-1} f_0  \|_{p,\alpha} \nonumber \\
   &\leq& c \  x^{k-1} \| \Lambda_\alpha^{k-1} f_0 \|_{p,\alpha}, \quad  x
   \in (0,+\infty).
  \end{eqnarray}
Hence by (4.1) et (4.2), we deduce that
\begin{equation*}
 \omega_{p,\alpha}^k(x,f) \leq c \  x^{k-1} K_{p,\alpha}^k(x,f),
\end{equation*}
then, $ f \in \mathcal{B}^k\mathcal{D}_{p,q}^{\beta,\alpha}.$\\
 \indent Let prove now the inclusion $\mathcal{B}^k\mathcal{D}_{p,q}^{\beta,\alpha}
  \subset \mathcal{K}^k\mathcal{D}_{p,q}^{\beta,\alpha}$. For $ f \in \mathcal{B}^k\mathcal{D}_{p,q}^{\beta,\alpha}$, we
  take for $ x\in (0,+\infty)$
  $$ f_1 = \frac{1}{b_k(x)}  \;  I_k(x,f). $$
  Using (3.8), we obtain
\begin{eqnarray}
  x \lVert \Lambda_\alpha^k f_1 \rVert_{p,\alpha} &\leq& \ x \; |b_k(x)|^{-1}\omega_{p,\alpha}^k(x,f)\nonumber\\
&\leq& c \; \frac{\omega_{p,\alpha}^k(x,f)}{x^{k-1}}.
  \end{eqnarray}
On the other hand, put $f_0= f-f_1 $, we can write using (3.5)
$$ f_0 = -\frac{1}{b_k(x)}  \int_{-x}^{x} \Theta_0 (x,y)
 \big(I_{k-1}(y,f) - b_{k-1}(y) f \big) A_\alpha(y) dy.$$
  From (3.1) and (3.8), we obtain
  $$ \Lambda_\alpha^{k-1} f_0 = -\frac{1}{b_k(x)} \int_{-x}^{x} \Theta_0 (x,y)R_k(y,f)A_\alpha(y) dy.$$
By Minkowski's inequality for integrals and (3.2), we get
\begin{eqnarray}
  \lVert \Lambda_\alpha^{k-1} f_0 \rVert_{p,\alpha}
&\leq & |b_k(x)|^{-1} \int_{-x}^{x} |\Theta_0 (x,y)| \ \lVert R_k(y,f) \rVert_{p,\alpha} A_\alpha(y) dy \nonumber \\
  & \leq & c \; x^{-k} \omega_{p,\alpha}^k(x,f)\int_{-x}^{x} |\Theta_0 (x,y)| \ A_\alpha(y) dy \nonumber \\
 &\leq & c \; \frac{\omega_{p,\alpha}^k(x,f)}{x^{k-1}}.
\end{eqnarray}
By (4.3) et (4.4), we deduce that
\begin{equation*}
 K_{p,\alpha}^k(x,f) \leq c\
 \frac{\omega_{p,\alpha}^k(x,f)}{x^{k-1}},
\end{equation*}
then, $f \in \mathcal{K}^k\mathcal{D}_{p,q}^{\beta,\alpha}$. Our
theorem is proved.
 \end{proof}

In order to establish that
$\widetilde{\mathcal{B}}^k\mathcal{D}_{p,q}^{\beta,\alpha}=
\mathcal{C}_{p,q}^{k,\beta,\alpha}$, we need to prove some useful
lemmas.
 \begin{lem}
Let $k=1,2,...,$ $1\leq p< +\infty$, $\phi \in
\mathcal{S}_\ast(\mathbb{R})$ such that
$\displaystyle\int_0^{+\infty}x^{2i}\phi(x)d\mu_\alpha(x)=0 $, for
all $i\in \{0,1,...,[\frac{k-1}{2}]\}$ and $r>0$, then there exists
a constant $c>0$ such that for all $f\in \mathcal{E}(\mathbb{R})\cap
L^p(\mu_\alpha) $ satisfying $\Lambda_\alpha^{k-1} f \in
L^p(\mu_\alpha) $ and $t>0$, we have
\begin{eqnarray} \|\phi_t \ast_\alpha f\|_{p,\alpha} \leq
c \int_0^{+\infty}\min\Big\{\Big(\frac{x}{t}\Big)^{2(\alpha+1)},
\Big(\frac{t}{x}\Big)^{r}\Big\}\,\|R_{k}(x,f)+R_{k}(-x,f)\|_{p,\alpha}\frac{dx}{x}.\end{eqnarray}
\end{lem}
\begin{proof} Let $t>0$, we have for $i\in \{0,1,...,[\frac{k-1}{2}]\}$,
  \begin{eqnarray}\int_0^{+\infty}x^{2i}\phi(x)d\mu_\alpha(x)=0\;
  \Longrightarrow\int_0^{+\infty}x^{2i}\phi_t(x)d\mu_\alpha(x)=0, \end{eqnarray}
  and \begin{eqnarray*}(\phi_t \ast_\alpha
  f)(y)&=&\int_{\mathbb{R}}\phi_t(x)\tau_y(f)(-x)d\mu_\alpha(x)\\&=&\int_{\mathbb{R}}\phi_t(x)\tau_y(f)(
  x)d\mu_\alpha(x).\end{eqnarray*} Then using (2.3), (3.9), (4.6) and Proposition 2.1,
   we can write for $y\in\mathbb{R}$ \begin{eqnarray*}2(\phi_t \ast_\alpha
  f)(y)&=&\int_{\mathbb{R}}\phi_t(x)\Big(\tau_y(f)(x)+\tau_y(f)(-x)-2\sum_{i=0}^{[\frac{k-1}{2}]} b_{2i}(x)
  \Lambda_\alpha^{2i} f(y)
    \Big)d\mu_\alpha(x)\\&=&
  2\int_0^{+\infty}\phi_t(x)\Big(\tau_x(f)(y)+\tau _{-x}(f)(y)-2\sum_{i=0}^{[\frac{k-1}{2}]}
   b_{2i}(x)\Lambda_\alpha^{2i} f(y)\Big)d\mu_\alpha(x)\\&=&2\int_0^{+\infty}\phi_t(x)
   \big(R_{k}(x,f)(y)+R_{k}(-x,f)(y)\big)d\mu_\alpha(x).\end{eqnarray*}
By Minkowski's inequality for integrals, we obtain\begin{eqnarray}
\;\|\phi_t \ast_\alpha f\|_{p,\alpha}&\leq&
 \int_0^{+\infty}|\phi_t(x)|\;\|R_{k}(x,f)+R_{k}(-x,f)
 \|_{p,\alpha}d\mu_\alpha(x)\nonumber\\&\leq&
 c\int_0^{+\infty}\Big(\frac{x}{t}\Big)^{2(\alpha+1)}
 \Big|\phi\Big(\frac{x}{t}\Big) \Big|\;\|R_{k}(x,f)+R_{k}(-x,f)
 \|_{p,\alpha}\frac{dx}{x}\\&\leq&
 c\int_0^{+\infty}\Big(\frac{x}{t}\Big)^{2(\alpha+1)} \|R_{k}(x,f)+R_{k}(-x,f)\|_{p,\alpha}\frac{dx}{x}. \end{eqnarray}
 On the other hand, since $\phi \in  \mathcal{S}_\ast( \mathbb{R})$, then from (4.7)
 and for $r>0$ there exists a constant $c$ such that
 \begin{eqnarray} \|\phi_t \ast_\alpha f\|_{p,\alpha}\leq
 c\int_0^{+\infty}\Big(\frac{t}{x}\Big)^{r} \;\|R_{k}(x,f)+R_{k}(-x,f)\|_{p,\alpha}\frac{dx}{x}
  \;.\end{eqnarray}
 Using (4.8) and (4.9), we deduce our result.
\end{proof}
\begin{exemp} According to (\cite{ro2}, Example 3.3,(2)), the
generalized Hermite polynomials on $\mathbb{R}$, denoted by
$H_n^{\alpha+\frac{1}{2}}$, $n\in \mathbb{N}$ are orthogonal with
respect to the measure $e^{-x^2}d\mu_\alpha(x)$ and can be written
as
$$H_{2n}^{\alpha+\frac{1}{2}}(x)=(-1)^n 2^{2n}n!\,L_n^\alpha(x^2)
\quad\mbox{and}\quad H_{2n+1}^{\alpha+\frac{1}{2}}(x)=(-1)^n
2^{2n+1}n!\,xL_n^{\alpha+1}(x^2),$$ where the $L_n^\alpha$ are the
Laguerre polynomials of index $\alpha\geq -\frac{1}{2}$, given by
$$L_n^\alpha(x)=\frac{1}{n!}\, x^{-\alpha}\,e^x
\frac{d^n}{dx^n}\Big(x^{n+\alpha}e^{-x}\Big).$$ For $k=1,2,...,$ fix
any positive integer $n_0>[\frac{k-1}{2}]$ and take for example the
function defined on $\mathbb{R}$ by $\phi(x)=H_{2
n_0}^{\alpha+\frac{1}{2}}(x)\,e^{-x^2}$. Put $P_i(x)=x^{2i}$ for
$i\in \{0,1,...,[\frac{k-1}{2}]\}$, since $P_i\in$
$span_\mathbb{R}\,\{H_p^{\alpha+\frac{1}{2}},
p=0,1,...,2[\frac{k-1}{2}]\}$, then we can assert that $ \phi \in
\mathcal{S}_\ast(\mathbb{R})$ and satisfy
$\displaystyle\int_0^{+\infty}x^{2i}\,\phi(x)\,d\mu_\alpha(x)=0 .$
\end{exemp}
\begin{lem} Let $k=1,2,...,$ $1<p< +\infty$ and $ \phi \in \mathcal{S}_\ast(\mathbb{R})$ such that
$\displaystyle\int_0^{+\infty}x^{2i}\phi(x)d\mu_\alpha(x)=0 $, for
all $i\in \{0,1,...,[\frac{k-1}{2}]\}$, then there exists a constant
$c>0$ such that for all $f\in \mathcal{E}(\mathbb{R})$ satisfying
$\Lambda_\alpha^{2i} f \in L^p(\mu_\alpha),$ $0\leq i \leq
[\frac{k-1}{2}]$ and $x>0$, we have
\begin{eqnarray}\|R_{k}(x,f)+
R_{k}(-x,f)\|_{p,\alpha} \leq c \int_0^{+\infty}
\min\Big\{\Big(\frac{x}{t}\Big)^{k-1},\Big(\frac{x}{t}\Big)^{k}\Big\}\|\phi_t
\ast_\alpha f\|_{p,\alpha}
\frac{dt}{t}\;.
\end{eqnarray}
\end{lem}
\begin{proof} Put for
$0<\varepsilon<\delta<+\infty$
$$
f_{\varepsilon,\delta}(y)=\int_\varepsilon^\delta
 (\phi_t \ast_\alpha \phi_t \ast_\alpha f)(y)\frac{dt}{t}\;,\;\;\;y\in
 \mathbb{R}.$$
 Then for $i \in \mathbb{N}$, we have
$$
\Lambda_{\alpha}^{2i}f_{\varepsilon,\delta}(y)=\int_\varepsilon^\delta
 (\Lambda_{\alpha}^{2i}\phi_t \ast_\alpha  \phi_t \ast_\alpha f)(y)\frac{dt}{t}\;,\;\;\;y\in \mathbb{R}.$$
From the integral representation of $\tau_x,$ we obtain by
interchanging the orders of integration and (2.7),
   \begin{eqnarray*}\tau_x(f_{\varepsilon,\delta})(y)&=&\int_\varepsilon^\delta
 \tau_x(\phi_t \ast_\alpha \phi_t \ast_\alpha f)(y)\frac{dt}{t}\\ &=&\int_\varepsilon^\delta
 (\tau_x (\phi_t) \ast_\alpha \phi_t \ast_\alpha f)(y)\frac{dt}{t} \;,\;\;y\in \mathbb{R},\;x\in(0,+\infty),\end{eqnarray*} so we
 can write for $x\in(0,+\infty)$ and $y\in \mathbb{R}$,\\ $(R_{k}(x,f_{\varepsilon,\delta})+R_{k}(-x,f_{\varepsilon,\delta}))(y) =
 \displaystyle \int_\varepsilon^\delta
 \big[\big(\tau_x(\phi_t)+\tau_{-x}(\phi_t)-2 \sum_{i=0}^{[\frac{k-1}{2}]} b_{2i}(x)
 \Lambda_{\alpha}^{2i}\phi_t\big) \ast_\alpha \phi_t \ast_\alpha f\big](y)\frac{dt}{t}\,.$
 Using the Minkowski's inequality for integrals and (2.6), we get\\
$\|(R_{k}(x,f_{\varepsilon,\delta})+R_{k}(-x,f_{\varepsilon,\delta}))\|_{p,\alpha}$
 \begin{eqnarray}
 &\leq&\int_\varepsilon^\delta
 \|(\tau_x(\phi_t)+\tau_{-x}(\phi_t)-2 \displaystyle\sum_{i=0}^{[\frac{k-1}{2}]} b_{2i}(x)\Lambda_{\alpha}^{2i}\phi_t)
  \ast_\alpha \phi_t \ast_\alpha f\|_{p,\alpha}\frac{dt}{t}\nonumber \\&\leq& c\int_\varepsilon^\delta
 \| (\tau_x(\phi_t)+\tau_{-x}(\phi_t)-2 \displaystyle\sum_{i=0}^{[\frac{k-1}{2}]} b_{2i}(x)\Lambda_{\alpha}^{2i}\phi_t) \|_{1,\alpha}
 \|\phi_t \ast_\alpha f\|_{p,\alpha}\frac{dt}{t}\nonumber\\
 &=& c\int_\varepsilon^\delta
\|R_{k}(x,\phi_{t})+R_{k}(-x,\phi_{t})\|_{1,\alpha}
 \|\phi_t \ast_\alpha f\|_{p,\alpha}\frac{dt}{t}\;.\end{eqnarray}
 For $x,\;t\in (0,+\infty)$, we have \\$ \|R_{k}(x,\phi_{t})+R_{k}(-x,\phi_{t})\|_{1,\alpha}$
 \begin{eqnarray} &=&\| \tau_x(\phi_t)+\tau_{-x}(\phi_t)-2 \displaystyle \sum_{i=0}^{[\frac{k-1}{2}]} b_{2i}(x)\Lambda_{\alpha}^{2i}\phi_t
 \|_{1,\alpha}\nonumber\\
 &=&\int_{\mathbb{R}}\Big|\Big(\int_{\mathbb{R}}\phi_t(z)(d\gamma_{x,y}(z)+d\gamma_{-x,y}(z))\Big)-2 \sum_{i=0}^{[\frac{k-1}{2}]}
 b_{2i}(x)\Lambda_{\alpha}^{2i}\phi_t(y)\Big|d\mu_\alpha(y)\nonumber\\
 &=&\int_{\mathbb{R}}\Big|\Big(\int_{\mathbb{R}}\phi\big( \frac{z}{t}\big)(d\gamma_{x,y}(z)+d\gamma_{-x,y}(z))\Big)-2\sum_{i=0}^{[\frac{k-1}{2}]} b_{2i}\big(\frac{x}{t}\big)\Lambda_{\alpha}^{2i}\phi\big(
 \frac{y}{t}\big)\Big|t^{-2(\alpha+1)}d\mu_\alpha(y)\;.\end{eqnarray}
 By (2.2) and the change of variable $z'=\frac{z}{t}$, we have
 $\displaystyle W_\alpha(x,y,z't)\;t^{2(\alpha+1)}=W_\alpha( \frac{x}{t}, \frac{y}{t},z'),$
 which implies that
  $\displaystyle d\gamma_{x,y}(z)=d\gamma_{\frac{x}{t},\frac{y}{t}}(z')\,.$
  Hence from (4.12), we obtain\\  $ \|R_{k}(x,\phi_{t})+R_{k}(-x,\phi_{t})\|_{1,\alpha}$
   \begin{eqnarray}
   &=&\int_{\mathbb{R}}\Big|\Big(\int_{\mathbb{R}}\phi( z')(d\gamma_{\frac{x}{t},\frac{y}{t}}(z')
  +d\gamma_{\frac{-x}{t},\frac{y}{t}}(z'))\Big)-2\sum_{i=0}^{[\frac{k-1}{2}]} b_{2i}\big(\frac{x}{t}\big)\Lambda_{\alpha}^{2i}\phi \big(\frac{y}{t}\big)\big|t^{-2(\alpha+1)}d\mu_\alpha(y)
 \nonumber\\ &=&\int_\mathbb{R}\Big|\Big(\tau_\frac{x}{t}(\phi)(\frac{y}{t})+\tau_\frac{-x}{t}(\phi)(\frac{y}{t})\Big)t^{-2(\alpha+1)}
  -2\big(\sum_{i=0}^{[\frac{k-1}{2}]} b_{2i}\big(\frac{x}{t}\big)\Lambda_{\alpha}^{2i}\phi \big)_t(y) \Big|d\mu_\alpha(y)\nonumber\\
  &=& \Big\|\Big(\tau_\frac{x}{t}(\phi) +\tau_\frac{-x}{t}(\phi)
     -2\sum_{i=0}^{[\frac{k-1}{2}]}
     b_{2i}\big(\frac{x}{t}\big)\Lambda_{\alpha}^{2i}\phi \Big)_{t}
   \Big\|_{1,\alpha}\nonumber \\&=& \Big\| \tau_\frac{x}{t}(\phi) +\tau_\frac{-x}{t}(\phi)
     -2 \sum_{i=0}^{[\frac{k-1}{2}]} b_{2i}\big(\frac{x}{t}\big)\Lambda_{\alpha}^{2i}\phi
   \Big\|_{1,\alpha}\nonumber\\
   &=& \Big\|
R_{k}(\frac{x}{t},\phi)+R_{k}(\frac{-x}{t},\phi)
 \Big\|_{1,\alpha} . \end{eqnarray}
  Since $\phi \in
\mathcal{S}_\ast(\mathbb{R})$, then using (2.10) and (3.3), we can
assert that
$$\Big\| R_{k}(\frac{x}{t},\phi)+R_{k}(\frac{-x}{t},\phi)
 \Big\|_{1,\alpha} \leq
c\;\big(\frac{x}{t}\big)^{k}\|\Lambda_{\alpha}^{k}\phi\|_{1,\alpha}\leq
c\;\big(\frac{x}{t}\big)^{k}\;,\;\; \quad$$
    on the other hand, by (3.4) we have $$\Big\|
R_{k}(\frac{x}{t},\phi)+R_{k}(\frac{-x}{t},\phi)
 \Big\|_{1,\alpha} \leq c\;\big(\frac{x}{t}\big)^{k-1}\|\Lambda_{\alpha}^{k-1}\phi\|_{1,\alpha}\leq
c\;\big(\frac{x}{t}\big)^{k-1}, $$
     then we get,\begin{eqnarray}\Big\| R_{k}(\frac{x}{t},\phi)+R_{k}(\frac{-x}{t},\phi) \Big\|_{1,\alpha}\leq c\;\min \Big\{\big(\frac{x}{t}\big)^{k-1},\big(\frac{x}{t}\big)^{k}\Big\}.\qquad\quad \;\end{eqnarray}
     From (4.11), (4.13) and (4.14), we obtain
 \begin{eqnarray} \|(R_{k}(x,f_{\varepsilon,\delta})+R_{k}(-x,f_{\varepsilon,\delta}))\|_{p,\alpha}
  \leq c \int_\varepsilon^{\delta}
\min\Big\{\big(\frac{x}{t}\big)^{k-1},\big(\frac{x}{t}\big)^{k}\Big\}\|\phi_t
\ast_\alpha f\|_{p,\alpha} \frac{dt}{t}\;.\end{eqnarray}
 Note that $\Lambda_{\alpha}^{2i}\phi\ast_\alpha\phi\in\mathcal{S_*}( \mathbb{R}).$ By
 (2.1) and
 (2.7), we have
\begin{eqnarray*}\int_\mathbb{R}(\Lambda_{\alpha}^{2i}\phi\ast_\alpha
\phi)(x)|x|^{2\alpha+1}dx&=&2^{\alpha+1}
\Gamma(\alpha+1)\mathcal{F}_\alpha(\Lambda_{\alpha}^{2i}\phi\ast_\alpha
\phi)(0)\\&=&2^{\alpha+1}
\Gamma(\alpha+1)\mathcal{F}_\alpha(\Lambda_{\alpha}^{2i}\phi)(0)\mathcal{F}_\alpha(\phi)(0)\\&=&2^{\alpha+1}
\Gamma(\alpha+1) \mathcal{F}_\alpha(\Lambda_{\alpha}^{2i}\phi)(0)
\int_\mathbb{R} \phi(z)d\mu_\alpha(z)=0.
\end{eqnarray*} Since $\Lambda_{\alpha}^{2i}\phi\ast_\alpha\phi$ is in the Schwartz space $\mathcal{S}( \mathbb{R})$, we have $$\int_\mathbb{R}|log|x||\;|\Lambda_{\alpha}^{2i}\phi\ast_\alpha\phi(x)|\;|x|^{2\alpha+1}dx<+\infty.$$
 Then, by Calder\'on's reproducing formula related to the Dunkl operator (see \cite{mou.T}, Theorem 3), we have
 $$\lim_{\varepsilon\rightarrow0,\;\delta\rightarrow +\infty}
 \Lambda_{\alpha}^{2i}f_{\varepsilon,\delta} = c \;\Lambda_{\alpha}^{2i}f\;,\;\;\;in\;L^p(\mu_\alpha)\;,$$
 hence from (4.15), we deduce our result.
 \end{proof}
\begin{thm} Let $0 <\beta <1$, $ k=1,2,...,$ $ 1 < p < +\infty $ and $ 1 \leq q \leq +\infty$, then we have
 \begin{eqnarray*}
 \widetilde{\mathcal{B}}^k\mathcal{D}^{p,q}_{\beta,\alpha}=
   \mathcal{C}_{p,q}^{k,\beta,\alpha},\end{eqnarray*} and for $p=1$,
   we have only $\widetilde{\mathcal{B}}^k\mathcal{D}^{1,q}_{\beta,\alpha}\subset
   \mathcal{C}_{1,q}^{k,\beta,\alpha}.$
\end{thm}
\begin{proof} Assume $f\in
\widetilde{\mathcal{B}}^k\mathcal{D}^{p,q}_{\beta,\alpha}$ for
$ 1 \leq p < +\infty $, $ 1 \leq q \leq +\infty$ and $r>\beta+k-1$.\\

$\bullet$ Case $q=1$. By (4.5) and Fubini's theorem, we have
\begin{eqnarray*}\displaystyle{\int_0^{+\infty}\frac{\|f \ast_\alpha \phi_t
\|_{p,\alpha}}{t^{\beta+k-1}}\frac{dt}{t}}
  &\leq& c \int_0^{+\infty}\int_0^{+\infty}\min\Big\{\Big(\frac{x}{t}\Big)^{2(\alpha+1)},
\Big(\frac{t}{x}\Big)^{r}\Big\}\widetilde{\omega}_{p,\alpha}^k(x,f)t^{-\beta-k}dt\frac{dx}{x}\\
 &\leq& c \int_0^{+\infty}\widetilde{\omega}_{p,\alpha}^k(x,f)\Big(\int_0^{+\infty}
 \min\Big\{\Big(\frac{x}{t}\Big)^{2(\alpha+1)},
\Big(\frac{t}{x}\Big)^{r}\Big\}t^{-\beta-k}dt\Big)\frac{dx}{x}\\
 &\leq& c \int_0^{+\infty}\widetilde{\omega}_{p,\alpha}^k(x,f)\Big(x^{-r}\int_0^xt^{r-\beta-k}dt+x^{2(\alpha+1)}
 \int_x^{+\infty}t^{-\beta-k-2\alpha-2}dt\Big)\frac{dx}{x}\\
  &\leq&c \int_0^{+\infty}\frac{\widetilde{\omega}_{p,\alpha}^k(x,f)}{x^{\beta+k-1}}\frac{dx}{x}<+\infty,\end{eqnarray*}hence
 $f\in \mathcal{C}_{p,1}^{k,\beta,\alpha}$.

$\bullet$ Case $q=+\infty$. By (4.5), we have
\begin{eqnarray*} \|\phi_t \ast_\alpha f\|_{p,\alpha} & \leq & c\;
\Big(\int_0^{t}\Big(\frac{x}{t}\Big)^{2(\alpha+1)}\widetilde{\omega}_{p,\alpha}^k(x,f)\frac{dx}{x}
+\int_t^{+\infty}\Big(\frac{t}{x}\Big)^{r}\widetilde{\omega}_{p,\alpha}^k(x,f)\frac{dx}{x}\Big)
\\&\leq& c
\sup_{x\in(0,+\infty)}\frac{\widetilde{\omega}_{p,\alpha}^k(x,f)}{x^{\beta+k-1}}
\Big(t^{-2(\alpha+1)}\int_0^tx^{2\alpha+\beta+k}dx
+t^r\int_t^{+\infty}x^{\beta+k-r-2}dx\Big)\\&\leq&
c\;t^{\beta+k-1}\sup_{x\in(0,+\infty)}\frac{\widetilde{\omega}_{p,\alpha}^k(x,f)}{x^{\beta+k-1}},\end{eqnarray*}
then we deduce that $f\in\mathcal{C}_{p,\infty}^{k,\beta,\alpha}$.

$\bullet$ Case $1<q<+\infty$. By (4.5) again, we have for $t>0$
$$\frac{\|\phi_t \ast_\alpha f\|_{p,\alpha}}{t^{\beta+k-1}}  \leq c
\int_0^{+\infty}\Big(\frac{x}{t}\Big)^{\beta+k-1}
\min\Big\{\Big(\frac{x}{t}\Big)^{2(\alpha+1)},
\Big(\frac{t}{x}\Big)^{r}\Big\}\frac{\widetilde{\omega}_{p,\alpha}^k(x,f)}{x^{\beta+k-1}}\frac{dx}{x}
\;.$$ Put $\displaystyle{L(x,t)=\Big(\frac{x}{t}\Big)^{\beta+k-1}
\min\Big\{\Big(\frac{x}{t}\Big)^{2(\alpha+1)},
\Big(\frac{t}{x}\Big)^{r}\Big\}}$ and
 $\displaystyle{q'=\frac{q}{q-1}}$ the conjugate of $q$. Since $$\displaystyle{\int_0^{+\infty} L(x,t) \frac{dx}{x}}
 =t^{-\beta-k-2\alpha-1}\int_0^t
x^{\beta+k+2\alpha}dx+
 t^{-\beta-k+r+1}\int_t^{+\infty} x^{\beta+k-r-2}dx\leq c,$$ we can write using H\"older's inequality,
\begin{eqnarray*} \frac{\|\phi_t \ast_\alpha
f\|_{p,\alpha}}{t^{\beta+k-1}} &\leq&c
\int_0^{+\infty}(L(x,t))^{\frac{1}{q'}}\Big((L(x,t))^{\frac{1}{q}}\frac{\widetilde{\omega}_{p,\alpha}^k(x,f)}{x^{\beta+k-1}}\Big)\frac{dx}{x}
\\&\leq&
c\;\Big(\int_0^{+\infty}L(x,t)\Big(\frac{\widetilde{\omega}_{p,\alpha}^k(x,f)}{x^{\beta+k-1}}\Big)^q\frac{dx}{x}\Big)^{\frac{1}{q}}.\end{eqnarray*}
By the fact that $$ \int_0^{+\infty}L(x,t)\frac{dt}{t} =
x^{\beta+k-r-1}\int_0^x t^{-\beta-k+r}dt+
 x^{\beta+k+2\alpha+1}\int_x^{+\infty} t^{-\beta-k-2\alpha-2}dt\leq c,$$ we get using Fubini's theorem,
\begin{eqnarray*} \int_0^{+\infty}\Big(\frac{\|\phi_t \ast_\alpha
f\|_{p,\alpha}}{t^{\beta+k-1}}\Big)^q \frac{dt}{t}
&\leq& c \int_0^{+\infty}\Big(\frac{\widetilde{\omega}_{p,\alpha}^k(x,f)}{x^{\beta+k-1}}\Big)^q\Big(\int_0^{+\infty}L(x,t)\frac{dt}{t}\Big)\frac{dx}{x}\\
&\leq& c
\int_0^{+\infty}\Big(\frac{\widetilde{\omega}_{p,\alpha}^k(x,f)}{x^{\beta+k-1}}\Big)^q
\frac{dx}{x}< +\infty ,
\end{eqnarray*} which proves the result.\\
Assume now $f\in\mathcal{C}_{p,q}^{k,\beta,\alpha}$ for $1<
p<+\infty$ and $1\leq q\leq+\infty.$

$\bullet$ Case $q=1$. By (4.10) and Fubini's theorem, we have
\begin{eqnarray*} \int_0^{+\infty}\frac{\widetilde{\omega}_p^\alpha(f)(x)}{x^{\beta+k-1}}\frac{dx}{x}
 &\leq & c \int_0^{+\infty}\int_0^{+\infty}
\min\Big\{\big(\frac{x}{t}\big)^{k-1},\big(\frac{x}{t}\big)^{k}\Big\}\|\phi_t
\ast_\alpha f\|_{p,\alpha}x^{-\beta-k} \frac{dt}{t}dx\\&\leq&
c\int_0^{+\infty}\|\phi_t \ast_\alpha
f\|_{p,\alpha}\Big(\int_0^{+\infty}
\min\Big\{\big(\frac{x}{t}\big)^{k-1},\big(\frac{x}{t}\big)^{k}\Big\}x^{-\beta-k}dx\Big)
\frac{dt}{t}\\&\leq& c\int_0^{+\infty}\|\phi_t \ast_\alpha
f\|_{p,\alpha}\Big(\frac{1}{t^k}\int_0^t
x^{-\beta}dx+\frac{1}{t^{k-1}}\int_t^{+\infty}
x^{-\beta-1}dx\Big)\frac{dt}{t}\\&\leq& c
\int_0^{+\infty}\frac{\|\phi_t \ast_\alpha
f\|_{p,\alpha}}{t^{\beta+k-1}}\frac{dt}{t}<+\infty,\end{eqnarray*}
then we obtain the result.

$\bullet$ Case $q=+\infty$. By (4.10), we get
\begin{eqnarray*} \widetilde{\omega}_p^\alpha(f)(x) &\leq& c\;
\Big(\int_0^{x}\big(\frac{x}{t}\big)^{k-1}\|\phi_t \ast_\alpha
f\|_{p,\alpha}\frac{dt}{t}+
 \int_x^{+\infty} \big(\frac{x}{t}\big)^{k} \|\phi_t \ast_\alpha f\|_{p,\alpha}
\frac{dt}{t}\Big)\\&\leq& c \sup_{t\in(0,+\infty)}\frac{\|\phi_t
\ast_\alpha f\|_{p,\alpha}}{t^{\beta+k-1}}
\Big(x^{k-1}\int_0^{x}t^{\beta-1}dt + x^k
\int_x^{+\infty}t^{\beta-2} dt \Big)\\&
 \leq& c \;x^{\beta+k-1} \sup_{t\in(0,+\infty)}\frac{\|\phi_t \ast_\alpha f\|_{p,\alpha}}
{t^{\beta+k-1}},\end{eqnarray*} so, we deduce that
$f\in\widetilde{\mathcal{B}}^k\mathcal{D}^{p,+\infty}_{\beta,\alpha}$.

$\bullet$ Case $1<q<+\infty$. By (4.10) again, we have for $x>0$
$$  \frac{\widetilde{\omega}_p^\alpha(f)(x)}{x^{\beta+k-1}} \leq c
\int_0^{+\infty}\Big(\frac{t}{x}\Big)^{\beta+k-1}
\min\Big\{\big(\frac{x}{t}\big)^{k-1},\big(\frac{x}{t}\big)^{k}\Big\}\frac{
\|\phi_t \ast_\alpha f\|_{p,\alpha}}{t^{\beta+k-1}}
\frac{dt}{t}\;.$$ Note that $$\Big(\frac{t}{x}\Big)^{\beta+k-1}
\min\Big\{\big(\frac{x}{t}\big)^{k-1},\big(\frac{x}{t}\big)^{k}\Big\}=\Big(\frac{t}{x}\Big)^\beta
\min\Big\{1,\frac{x}{t}\Big\}.$$ Put
$\displaystyle{G(x,t)=\Big(\frac{t}{x}\Big)^\beta
\min\Big\{1,\frac{x}{t}\Big\}}.$ Since
$$\displaystyle{\int_0^{+\infty} G(x,t) \frac{dt}{t}}
 =x^{-\beta}\int_0^x
t^{\beta-1}dt+
 x^{-\beta+1}\int_x^{+\infty} t^{\beta-2}dt\leq c,$$ then using H\"older's inequality, we can write
  \begin{eqnarray*}\frac{\widetilde{\omega}_p^\alpha(f)(x)}{x^{\beta+k-1}}
  &\leq& c \int_0^{+\infty}(G(x,t))^{\frac{1}{q'}}\Big((G(x,t))^{\frac{1}{q}}
\frac{ \|\phi_t \ast_\alpha f\|_{p,\alpha}}{t^{\beta+k-1}}\Big)
\frac{dt}{t}\\&\leq& c \;\Big(\int_0^{+\infty}  G(x,t) \Big( \frac{
\|\phi_t \ast_\alpha f\|_{p,\alpha}}{t^{\beta+k-1}}\Big)^q
\frac{dt}{t}\Big)^{\frac{1}{q}}.\end{eqnarray*} By the fact that
$$\int_0^{+\infty}G(x,t)\frac{dx}{x} =
t^{\beta-1}\int_0^t x^{-\beta}dx+
 t^{\beta}\int_t^{+\infty} x^{-\beta-1}dx
 \leq c,$$  we
get using Fubini's theorem,
\begin{eqnarray*}
\int_0^{+\infty}\Big(\frac{\widetilde{\omega}_p^\alpha(f)(x)}{x^{\beta+k-1}}\Big)^q
\frac{dx}{x}&\leq& c\int_0^{+\infty}\Big(\frac{ \|\phi_t \ast_\alpha
f\|_{p,\alpha}}{t^{\beta+k-1}}\Big)^q\Big(\int_0^{+\infty} G(x,t)
\frac{dx}{x}\Big)\frac{dt}{t}\\&\leq& c\int_0^{+\infty}\Big(\frac{
\|\phi_t \ast_\alpha f\|_{p,\alpha}}{t^{\beta+k-1}}\Big)^q
\frac{dt}{t} <+\infty,
\end{eqnarray*} thus the result is established.
\end{proof}

\end{document}